\documentclass[11pt]{article}

\usepackage{amsmath,amssymb,url}
\usepackage{enumitem} 
\usepackage{latexsym}
\usepackage[mathscr]{eucal}

\numberwithin{equation}{section}

\newtheorem{thm}{Theorem}[section]
\newtheorem{lem}[thm]{Lemma}
\newtheorem{pro}[thm]{Proposition}
\newtheorem{cor}[thm]{Corollary}
\newtheorem{dfn}[thm]{Definition}

\newtheorem{eg}[thm]{Example}

\newcommand{\ben}{\begin{enumerate}}
\newcommand{\een}{\end{enumerate}}

\newcommand{\bi}{\begin{itemize}}
\newcommand{\ei}{\end{itemize}}

\DeclareMathOperator{\dom}{dom}

\DeclareMathOperator{\Par}{\sf Par}

\newenvironment{proof}{\noindent \textbf{Proof.}\hspace{.7em}}
                   {\hfill $\Box$
                    \vspace{10pt}}

\newcommand{\ol}{\overline}

\newcommand{\mc}{\mathcal}

\begin{document}

\title{Override and restricted union for partial functions}
\author{Tim Stokes}

\date{}
\maketitle

\begin{abstract}
The {\em override} operation $\sqcup$ is a natural one in computer science, and has connections with other areas of mathematics such as hyperplane arrangements.  For arbitrary functions $f$ and $g$, $f\sqcup g$ is the
function with domain $\dom(f)\cup \dom(g)$ that agrees with $f$ on $\dom(f)$ and with $g$ on $\dom(g) \backslash \dom(f)$.  
Jackson and the author have shown that there is no finite axiomatisation of algebras of functions of signature $(\sqcup)$.  But adding operations (such as {\em update}) to this minimal signature can lead to finite axiomatisations.  For the functional signature $(\sqcup,\backslash)$ where $\backslash$ is set-theoretic difference, Cirulis has given a finite equational axiomatisation as subtraction o-semilattices.  Define $f\curlyvee g=(f\sqcup g)\cap (g\sqcup f)$ for all functions $f$ and $g$; this is the largest domain restriction of the binary relation $f\cup g$ that gives a partial function.  Now $f\cap g=f\backslash(f\backslash g)$ and $f\sqcup g=f\curlyvee(f\curlyvee g)$ for all functions $f,g$, so the signatures $(\curlyvee)$ and $(\sqcup,\cap)$ are both intermediate between $(\sqcup)$ and $(\sqcup,\backslash)$ in expressive power.  We show that each is finitely axiomatised, with the former giving a proper quasivariety and the latter the variety of associative distributive o-semilattices in the sense of Cirulis.
\end{abstract}

\medskip


\section{Introduction}   \label{sec:Intro}

Let $X$ and $Y$ be sets, and let $\Par(X,Y)$ denote the set of all partial functions from $X$ to $Y$; these are the functions with domain in $X$ and range in $Y$. 

In the theory of function algebras, one seeks to axiomatise algebras of (partial) functions under one or more operations or relations.  The elements of such an algebra are members of $\Par(X,Y)$ for some fixed $X,Y$, and the operations are from some given signature.  If composition is in the signature, it is normally assumed that $X=Y$.

Recall the following familiar purely set-theoretic operations on $\Par(X,Y)$:
\bi
\item intersection $\cap$, defined as the usual set-theoretic intersection of (the graphs of) $f$ and $g$;
\item difference $\backslash$, defined to be the usual set-theoretic difference of $f$ and $g$:
\[
f\backslash g=\{(x,y)\mid (x,y)\in f,(x,y)\not\in g\}.
\]
\ei

For function algebras, the signature of intersection gives nothing but semilattices, and difference gives so-called subtraction algebras (which have appeared in different forms in the literature, for example implication algebras \cite{abb1,abb2} and implicative BCK-algebras \cite{impl}).  

We are here mainly interested in two functional analogs of the set operation of union, which, unlike intersection and difference, do not have purely set-theoretic definitions.  To define them, if $f,g\in \Par(X,Y)$, denote by $g-f$ the restriction of $g$ to the complement of the domain of $f$ (a notation used in \cite{berendsen}).  Then we may define, for all functions $f,g\in \Par(X,Y)$,
\bi
\item {\em override} $\sqcup$, given by 
\[
f\sqcup g=f\cup (g-f);
\]
\item restricted union $\curlyvee$, given by
\[
f\curlyvee g=(f-g)\cup (f\cap g)\cup (g-f).
\]
\ei
Each of $f\sqcup g$ and $f\curlyvee g$ is a disjoint union of its constituent parts and hence gives a function when applied to two functions.  
Indeed, it is easy to see that $f\curlyvee g$ is the domain restriction of the binary relation $f\cup g$ to where it is single-valued (that is, a function).  Notably, $\sqcup$ is associative but not commutative, and $\curlyvee$ is the opposite.  Both are idempotent.

The {\em override} operation has received considerable attention in theoretical computer science and mathematics, especially when taken in combination with other operations such as $\cap$ and others; see \cite{berendsen}, \cite{CLS}, \cite{modrest} and \cite{overup}, as well as the earlier works \cite{LeechSBA} and \cite{NSL}.  (The symbol $\sqcup$ was used in \cite{modrest} and \cite{overup}, but other notations are often used.)  It is shown in \cite{modrest} that {\em override} has a close connection to the {\em if-then-else} construct.  As discussed in \cite{cirulis}, the operation is also of interest in the setting of so-called flat records, which are modelled as (partial) functions on a finite domain, and it also arises as the (opposite of the) overriding operator $\oplus$ in Z \cite{spivey}.  Finally, there are interesting connections with hyperplane arrangements \cite{MSS}.  Restricted union was defined and briefly discussed in Section 3.3 of \cite{modrest}, where it was noted to be equivalent in expressive power to {\em override} in the presence of intersection, but as far as the author knows, it has not been considered since.

In some previous work, such as \cite{berendsen} and \cite{CLS}, a different operation which is often called intersection but is based on the idea of domain restriction (and studied at least since the time of \cite{sch70}) was either part of the signature considered or derivable within it; we do not consider this operation  here.  The original motivation for the current work came from interest in functional signatures having actual intersection together with some functional approximant of union such as $\sqcup$ or $\curlyvee$.

There are straightforward connections between some of the signatures just discussed.  For example, $\cap$ can be expressed in terms of $\backslash$, as noted in the abstract.  But the two notions of union are also linked: $f\sqcup g=f\curlyvee(f\curlyvee g)$ (see Proposition \ref{curlyveeprops} to follow), and $f\curlyvee g=(f\sqcup g)\cap (g\sqcup f)$ (see Proposition \ref{veevee}), so $\curlyvee$ contains some information about intersection as well as $\sqcup$.  It follows that $\{\cap,\curlyvee\}$ is equivalent in expressive power to $\{\cap,\sqcup\}$, and that an axiomatisation of the functional algebras of one signature easily gives an axiomatisation of functional algebras of the other.

Note also that $\curlyvee$ is strictly more expressive than $\sqcup$: any collection $S$ of more than one partial function in $\Par(X,Y)$, all members of which are defined everywhere on $X$, will be such that $(S,\sqcup)$ is a left zero semigroup, but $s\curlyvee t$ will not be in $S$ unless $s=t$.  Moreover, $\curlyvee$ alone is strictly less expressive than $\cap$ plus $\curlyvee$ (equivalently, $\cap$ plus $\sqcup$), as a simple example shows.  Let $X=\{1,2,3,0\}$ with $S=\{f,g,h,\emptyset\}$ with 
\[
f=\{(1,0),(2,0)\}, g=\{(2,0),(3,0)\}, h=\{(1,0),(2,0),(3,0)\}.
\]  
Then $S$ is closed under $\curlyvee$, since $h=f\curlyvee g$, with all other cases of $\curlyvee$ trivial since for any other pair of elements, one is a subset of the other.  But of course $S$ is not closed under intersection since $f\cap g\not\in S$.  So functional algebras and semigroups closed under $\curlyvee$ are more general than those closed under both $\curlyvee$ and $\cap$.  

The signature $(\sqcup)$ is considered in \cite{overup}, where the functional algebras are shown to have no finite axiomatisation, and an infinite quasiequational axiomatisation is given.  By contrast, in Section 5 of \cite{cirulis}, Cirulis axiomatises functional algebras of signature $(\sqcup,\backslash)$, as the finitely based variety of subtraction o-semilattices (see Theorem 5.6 there).  The two signatures we consider here lie between these two known cases in terms of expressive power.

In \cite{cirulis}, the author gives a structural description of the functional algebras having the signature $(\sqcup,\cap)$, but does not axiomatise them.  We do this here, showing them to be nothing but the variety of associative distributive o-semilattices defined in \cite{cirulis}.  It is straightforward to then enrich the signature to include difference and recover the finite equational axiomatisability result due to Cirulis in \cite{cirulis}, albeit with a different (though necessarily equivalent) equational axiomatisation to the one given in Theorem 5.6 there. But we also finitely axiomatise the signature $(\curlyvee)$, giving a proper quasivariety.  This covers off every signature consisting of some or all of the four functional operations $(\sqcup,\curlyvee,\cap,\backslash)$.  

A partial operation of domain-disjoint union of sets and partial functions is considered in \cite{H+Mc}, motivated by separation logic.  It was shown there that no finite axiomatisation for either set-based or functional algebras exists (though again, an axiomatisation was given).  This is consistent with the situation for the signature $(\sqcup)$ (noting that $\sqcup$ agrees with domain-disjoint union when the latter is defined).  
But it is perhaps surprising that the signature $(\curlyvee)$ admits a finite axiomatisation, since, like $\sqcup$, it agrees with the partial operation of union when the latter exists; it seems that the intersection-related information present in the definition of $\curlyvee$ is enough to allow this. 

In what follows, if $\Sigma$ is a particular signature comprising some of the operations $\sqcup,\curlyvee,\cap,\backslash$ considered above, we say an equational or quasiequational law in the signature $\Sigma$ is {\em functionally sound} if it is satisfied by $\Par(X,Y)$ for any non-empty set $X$; it follows that the law is satisfied by any subalgebra (under the relevant signature) of $\Par(X,Y)$.  We say an algebra of signature $\Sigma$ is {\em functional} if it is isomorphic to one whose elements are partial functions, and the operations have their standard functional interpretations as described above. 

In the section to follow, we begin by considering properties of {\em override} alone, noting that it is a left regular band operation and recording some basic properties of such bands that will be useful in what follows.  In Section \ref{sec:algebras}, we consider the associative distributive o-semilattices considered in \cite{cirulis}, where their defining laws are observed to be functionally sound.  We call these {\em ado-semilattices}, and reformulate their axioms somewhat for later use.  We  define $\curlyvee$-algebras in terms of some laws satisfied by $\curlyvee$ defined in an ado-semilattice via $a\curlyvee b=(a\sqcup b)\cap (b\sqcup a)$; it follows that every functional algebra of signature $(\curlyvee)$ is a $\curlyvee$-algebra.

Then in Section \ref{sec:complete}, we present our completeness proofs.  Most of the work is in showing that the class of $\curlyvee$-algebras axiomatises the functional algebras of signature $(\curlyvee)$, which is also shown to be a proper quasivariety.  We then build on this to prove that the finitely based variety of ado-semilattices axiomatises the functional algebras of signature $(\sqcup,\cap)$.  We conclude in Section \ref{sec:diff} by obtaining as a corollary a different finite equational axiomatisation for the functional signature $(\sqcup,\backslash)$.

Throughout, we write functions on the right of their arguments, so ``$xf$", rather than ``$f(x)$".

\section{Left regular bands}  \label{sec:lrb}

The operation of $\sqcup$ alone is rather badly behaved in terms of axiomatisation of functional algebras: by 
Corollary 4.1 of \cite{overup}, the class of functional algebras of signature $(\sqcup)$ has no finite axiomatisation in first order logic.  Indeed it is shown there to consist of precisely those algebras that arise from hyperplane arrangements, and is the proper quasivariety generated by the three-element band $\{1,e,f\}$ in which $1$ is an identity element appended to the left zero semigroup on $\{e,f\}$.  It was also shown in \cite{overup} that the addition to the signature of the operation of {\em update} as defined there and in earlier work including \cite{berendsen} and \cite{CLS} leads to a finite equational axiomatisation, answering a question posed in that earlier work.  It was shown in \cite{berendsen} that adding the operation {\em minus} to {\em override} (yielding a signature that subsumes {\em update}) gives a finite axiomatisation, at least for the equational theory generated by the functional algebras.  The more general finite axiomatisability of the functional algebras follows from the work of \cite{CLS}, itself reliant on the work of \cite{LeechSBA} on skew Boolean algebras.  We shall see very similar behavior for the cases considered here, in which the signature of $\sqcup$ is enriched in various natural ways that give finite axiomatisations.

We next note note some quite elementary properties satisfied by the {\em override} operation that will be useful in what follows.
  
Recall that $(S,\sqcup)$ is a {\em left regular band} if for all $a,b,c\in S$,
\ben 
\item $a\sqcup(b\sqcup c)=(a\sqcup b)\sqcup c$;
\item $a\sqcup a=a$;
\item $a\sqcup b=(a\sqcup b)\sqcup a$.
\een

The fact that the left regular band properties are all satisfied by {\em override} for partial functions has been noted by several authors; see \cite{berendsen} and \cite{cirulis}.  Note that in the first of these, the operation of {\em override} is denoted $\rhd$, whilst in the latter, it is denoted $\lhd$ but defined opposite to the way most other authors (including the current one) define it.

On the left regular band $S$, define
\bi
\item $a\leq b$ if $a\sqcup b=b$.
\item $a\lesssim b$ if $b\sqcup a=b$;
\ei
For the case of partial functions, $\leq$ is the partial order of set-theoretic inclusion of the sets of ordered pairs representing them, and $\lesssim$ is the domain inclusion quasiorder: $f\lesssim g$ if and only if $\dom(f)\subseteq \dom(g)$.  Note also that $\dom(f\sqcup g)=\dom(f)\cup \dom(g)$.

Most of the facts noted in the following are well-known, but good sources are \cite{sal} and \cite{mss}.

\begin{lem} \label{useful}
Let $S$ be a left regular band.
\ben
\item The relation $\leq$ is a partial order on $S$, and for all $a,b,c\in S$, if $b\leq c$ then $a\sqcup b\leq a\sqcup c$.
\item The relation $\lesssim$ is a quasiorder and for all $a,b,c,d\in S$,  if $a\lesssim b,c\lesssim d$ then $a\sqcup c\lesssim b\sqcup d$.
\item For all $a,b\in S$, $a\sqcup b\sim b\sqcup a$.
\item For all $a,b\in S$, $a\leq b$ implies $a\lesssim b$. 
\item If $d\lesssim a_1\sqcup i_1, a_2\sqcup i_2,\ldots,a_n\sqcup i_n$ in $S$, then letting $i=i_1\sqcup i_2\sqcup \ldots\sqcup i_n$, we have that $d\lesssim a_j\sqcup i$ for $j=1,2,\ldots, n$.
\een
\end{lem}
\begin{proof} (1) That $\leq$ is a partial order was noted in Section 2 of \cite{sal}, and the second part is Lemma 2.1 in \cite{mss}. 

(2) That $\lesssim$ is a quasiorder is again noted in Section 2 of \cite{sal}.  
further, if $a\lesssim b$ and $c\lesssim d$ then $b\sqcup a=b$ and $d\sqcup c=d$, so 
\[
(b\sqcup d)\sqcup (a\sqcup c)=b\sqcup d\sqcup b\sqcup a\sqcup d\sqcup c=b\sqcup d\sqcup b\sqcup d=b\sqcup d,
\]
and so $a\sqcup c\lesssim b\sqcup d$.

(3) above follows from (2) and (3) of Proposition 2.1 in \cite{sal}, and (4) here follows from (1) and (3) of that result.
%

(5) Assume that $d\lesssim a_1\sqcup i_1, a_2\sqcup i_2,\ldots,a_n\sqcup i_n$.  Then $a_j\sqcup i_j\sqcup d=a_j\sqcup i_j$, $j=1,2,\ldots,n$.  Hence for any $j\in \{1,2,\ldots,n\}$, 
\begin{align*}
a_j\sqcup i\sqcup d&= a_j\sqcup i_1\sqcup i_2\sqcup\cdots\sqcup i_n\sqcup d\\
&=a_j\sqcup i_1\sqcup i_2\sqcup\cdots\sqcup i_n\sqcup a_j\sqcup i_j\sqcup d\\
&=a_j\sqcup i_1\sqcup i_2\sqcup\cdots\sqcup i_n\sqcup a_j\sqcup i_j\\
&=a_j\sqcup i_1\sqcup i_2\sqcup\cdots\sqcup i_n\\
&=a_j\sqcup i,
\end{align*}
so $d\lesssim a_j\sqcup i$. 
\end{proof}

We remark that defining $a\leq b$ if and only if $ab=b$ for all $a,b$ in a band $S$ gives a partial order if and only if $S$ is left regular.  Moreover, distinct left regular bands structures on the same set can induce the same partial order, so the poset structure does not determine the left regular band operation.  (This contrasts with the case of the {\em override} operation in overriding nearlattices, considered in the next section.)  The following example to show this is well-known.

\begin{eg} Two non-isomorphic left regular band structures on the same set giving the same partial order.  \label{difflrb}

{\em Let $S=\{b,t,e,f\}$.  Consider the four element Boolean algebra structure on $S$, with bottom $b$, top $t$, $e\wedge f=b$, $e\vee f=t$; then $(S,\vee)$ is a semilattice and hence a left regular band.  Now define $\sqcup$ on $S$ by adjoining zero $t$ and identity $b$ to the left zero semigroup defined on $\{e,f\}$.  The partial order obtained from $(S,\vee)$ equals that obtained from $(S,\sqcup)$ but obviously these are non-isomorphic bands.
}
\end{eg}

\begin{dfn}
Suppose $S$ is a left regular band.  Let $I\subseteq S$ be non-empty.  We say $I$ is a {\em $\lesssim$-ideal} if it is a down-set under $\lesssim$, additionally satisfying the requirement that if $i,j\in I$ then $i\sqcup j\in I$.  For $a\in S$, we say the $\lesssim$-ideal is {\em relatively maximal (with respect to not containing $a$)} if $a\not\in I$ and there is no larger $\lesssim$-ideal of $S$ that does not contain $a$.  
\end{dfn}  

Every maximal $\lesssim$-ideal (defined as expected)
 is relatively maximal, although the converse may fail.

\begin{lem}  \label{relmax}
Let $S$ be a left regular band with $I$ a $\lesssim$-ideal of $S$.  Suppose $I$ is relatively maximal with respect to not containing $d\in S$, and $a\not\in I$.  Then there exists $i\in I$ such that $d\lesssim a\sqcup i$.
\end{lem}
\begin{proof}
Let $G=\{w\in S\mid w\lesssim a\sqcup i, i\in I\}$.  Then $G$ is a $\lesssim$-ideal since it is obviously down-closed with respect to $\lesssim$, and moreover if $w_1\lesssim a\sqcup i_1$ and $w_2\lesssim a\sqcup i_2$ for some $i_1,i_2\in I$, then 
\begin{align*}
a\sqcup i_2\sqcup i_1\sqcup (w_1\sqcup w_2) 
&=a\sqcup i_2\sqcup a\sqcup i_1\sqcup (w_1\sqcup w_2) \\
&=a\sqcup i_2\sqcup (a\sqcup i_1\sqcup w_1)\sqcup w_2 \\
&=a\sqcup i_2\sqcup (a\sqcup i_1)\sqcup w_2 \\
&=a\sqcup i_2\sqcup a\sqcup i_1\sqcup a\sqcup i_2\sqcup  w_2 \\
&=a\sqcup i_2\sqcup a\sqcup i_1\sqcup a\sqcup i_2\\
&=a\sqcup i_2\sqcup a\sqcup i_1\\
&=a\sqcup i_2\sqcup i_1
\end{align*}
and so because $i_2\sqcup i_1\in I$, we have $w_1\sqcup w_2\in G$.  Now $a\sqcup i\sqcup i=a\sqcup i$, so $i\lesssim a\sqcup i$ for all $i\in I$, and so $I\subseteq G$, and also since $a\lesssim a\sqcup i\sqcup a=a\sqcup i$ for each $i\in I$, we have $a\lesssim a\sqcup i$ and so $a\in G$.  So by relative maximality of $I$ with respect to not containing non-zero $d\in S$, we have that $d\in G$, and so $d\lesssim a\sqcup i$ for some $i\in I$.
\end{proof}

\section{Ado-semilattices and $\curlyvee$-algebras}  \label{sec:algebras}

\subsection{Ado-semilattices}

In \cite{cirulis}, an {\em overriding nearlattice} is defined to be a nearlattice (meaning a (meet)-semilattice $(S,\wedge)$ in which any two elements with an upper bound have a least upper bound) such that, for all $a,b\in S$, the following exists:
\[a\lhd b = sup\{x\in S : (x \leq a \mbox{ and }x,b\mbox{ share an upper bound), or }x \leq b \}.\]
If $a,b$ have an upper bound, then their least upper bound is $a\lhd b$.  Note that $a\lhd b$ in an overriding nearlattice is wholly determined by its structure as a semilattice. It is observed in \cite{cirulis} that every functional algebra of signature $(\sqcup,\cap)$ is an overriding nearlattice in which $a\wedge b=a\cap b$ and $a\lhd b=b\sqcup a$.

In \cite{cirulis}, the author went on to axiomatise those algebras $(S,\cap,\lhd)$ that arise from overriding nearlattices as above, as {\em semilattices with overriding}, abbreviated to {\em o-semilattices}.  Using our notation, with $a\sqcup b=b\lhd a$, o-semilattices may be defined as algebras $(L,\cap,\sqcup)$ such that $(L,\cap)$ is a semilattice, and such that, defining $\leq$ via $x\leq y\Leftrightarrow x=x\cap y$ as usual, we have for all $x,y,z\in L$:  
\ben
\item $x\leq x\sqcup y$;
\item $(x\cap y)\sqcup (y\cap z)\leq y$;
\item $x\sqcup y \leq x\sqcup(y\cap (x\sqcup y))$;
\item $x\cap z\leq (x\cap y)\sqcup z$.
\een
In what follows, we use this notation and these laws when we refer to o-semilattices.

Functional algebras of signature $(\sqcup,\cap)$ were amongst the main motivating examples of overriding nearlattices (and hence of o-semilattices) given in \cite{cirulis}.  The author described their structure: they are precisely the subdirect products of flat o-semilattices; see Theorem 5.2 there.  (A poset $(P,\leq)$ is {\em flat} if there is a smallest element $0\in P$ and all other elements are maximal; these are easily seen to be o-semilattices.)  He did not axiomatise them, but did give a finite equational axiomatisation of the functional algebras of signature $(\lhd,\backslash)$, as the class of overriding nearlattices with associative $\lhd$ in which the principal ideal (with respect to the meet-semilattice order $\leq$) generated by each element is a Boolean algebra.  He gave an equational axiomatisation in the language of $(\lhd,\backslash)$ as so-called {\em subtraction o-semilattices}; we return to these in Section \ref{sec:diff}.

An overriding nearlattice (hence an o-semilattice) is said to be {\em distributive} if the principal ideal generated by each element is a distributive lattice; so ``Boolean" implies ``distributive".  In terms of the signature $(\sqcup,\cap)$, this distributivity condition can be stated equationally as
\[
(a\cap d)\sqcup ((b\cap d)\cap (c\cap d))=((a\cap d)\sqcup (b\cap d))\cap ((a\cap d)\sqcup (c\cap d)).
\]
(Of course there is a dual form in terms of $\lhd$, which is the form given in \cite{cirulis}.) Every functional o-semilattice is distributive (since for partial functions contained in a given one, $\sqcup$ is the same as union).  But it is also associative, meaning that $(x\sqcup y)\sqcup z=x\sqcup (y\sqcup z)$.  

\begin{dfn}
An {\em ado-semilattice} is a distributive o-semilattice in which $\sqcup$ is associative.
\end{dfn}

If an o-semilattice is a lattice, (so that $\sqcup$ is simply join), then the distributivity condition is simply lattice distributivity, which is also sufficient for functional representability since every distributive lattice can be represented as subsets of a set with intersection and union as the meet and join.  So amongst o-semilattices which are lattices, distributivity is necessary and sufficient for functional representability.  We shortly show that this is true for general associative o-semilattices.  

Next, we note some laws satisfied by ado-semilattices, some of which we have seen but some of which are new; since the ado-semilattice laws are functionally sound, so must these derived laws be.

\begin{lem}  \label{eqado}
Every ado-semilattice $(S,\sqcup,\cap)$ is such that  $(S,\sqcup)$ is a left regular band, satisfying {\em extended distributivity} given by $a\sqcup (b\cap c)=(a\sqcup b)\cap (a\sqcup c)$.
\end{lem}
\begin{proof}
Most of the laws follow from properties listed in Lemma 2.2 of \cite{cirulis}; in particular, the left regular band laws are associativity together with properties (xiv) and (xvi), and $a\cap b=b\cap a\leq a$ follows from (xii).  

Next we prove extended distributivity.  First note that the third law $x\sqcup y \leq x\sqcup(y\cap (x\sqcup y))$ given above for o-semilattices can be strengthened to an equation since the reverse inclusion holds by (1) in Lemma \ref{useful}.  Hence, for all $a,b,c$,
\begin{align*}
a\sqcup (b\cap c)&=a\sqcup ((b\cap c)\cap(a\sqcup (b\cap c))\\
&\leq a\sqcup ((b\cap c)\cap(a\sqcup (b\cap c)\sqcup c)\\
&=a\sqcup ((b\cap c)\cap (a\sqcup c))\\
&\leq a\sqcup (b\cap c),
\end{align*}
so all are equal, and so in particular, 
\begin{align*}
a\sqcup (b\cap c)&=a\sqcup ((b\cap c)\cap (a\sqcup c))\\
&=(a\cap(a\sqcup c))\sqcup ((b\cap(a\sqcup c))\cap(c\cap(a\sqcup c)))\\
&=((a\cap(a\sqcup c))\sqcup (b\cap(a\sqcup c))\cap(a\cap(a\sqcup c))\sqcup (c\cap(a\sqcup c)))\\
&\mbox{ (by o-semilattice distributivity)}\\
&=(a\sqcup (b\cap(a\sqcup c)))\cap(a\sqcup (c\cap(a\sqcup c)))\\
&=(a\sqcup (b\cap(a\sqcup c)))\cap(a\sqcup c)\\
&\mbox{ (by the strengthened o-semilattice law (3))}\\
&=a\sqcup (b\cap (a\sqcup c))\mbox{  since $a\sqcup (b\cap (a\sqcup c))\leq a\sqcup(a\sqcup c)=a\sqcup c$}\\
&\mbox{ (upon using (1) of Lemma \ref{useful})}\\ 
&=a\sqcup ((a\sqcup c)\cap b).
\end{align*}
But then applying this result again to $a\sqcup (b\cap(a\sqcup c))$ gives that
\[ a\sqcup (b\cap c)=a\sqcup ((a\sqcup b)\cap(a\sqcup c))=(a\sqcup b)\cap(a\sqcup c),\]
since $a\leq a\sqcup b,a\sqcup c$, so $a\leq (a\sqcup b)\cap(a\sqcup c)$.
\end{proof}

The extended distributivity law can be used to replace the distributivity law for ado-semilattices since the former clearly implies the latter.  We can then give an alternative axiomatisation for the class of ado-semilattices in which the partial order determined by $\sqcup$ is primary.  

\begin{cor}  \label{altado}
The class of ado-semilattices may be defined by the following:
\ben
\item $\sqcup$ is a left regular band operation;
\item $\cap$ is semilattice meet with respect to the partial order determined by $\sqcup$;
\item $a\sqcup(b\cap c)=(a\sqcup b)\cap (a\sqcup c)$;
\item  $a\cap c\leq (a\cap b)\sqcup c$.
\een
\end{cor}
\begin{proof}
We have seen in Lemma \ref{eqado} that the above laws follow from those for ado-semilattices.  Conversely, these laws imply those for ado-semilattices (expressed in terms of $\sqcup$): the first two laws of o-semilattices follow easily, and the third follows from the third extended distributivity law above, as does ado-semilattice distributivity, and associativity of $\sqcup$ is immediate.
\end{proof}

We next establish some further facts about ado-semilattices.

\begin{lem}  \label{goodies}
Suppose $(S,\sqcup,\cap)$ is an ado-semilattice.  Then for all $a,b,c\in S$, 
\ben
\item $a\sqcup b=b\sqcup a$ if and only if $a,b$ have an upper bound;
\item $a\cap ((a\cap b)\sqcup c)=a\cap ((a\cap c)\sqcup b)$.
\een
\end{lem}
\begin{proof}
Now if $a\sqcup b=b\sqcup a$ then $a,b\leq a\sqcup b$.  Conversely if $a,b\leq c$ then $a\sqcup b=b\sqcup a$ because $a\sqcup b$ is the join of $a,b$ in the principal ideal determined by $c$.this establishes (1).  For (2), using extended distributivity we obtain 
\[
(a\cap b)\sqcup (a\cap c)=((a\cap b)\sqcup a)\cap ((a\cap b)\sqcup c)=a\cap((a\cap b)\sqcup c),
\] 
and similarly $(a\cap c)\sqcup (a\cap b)=a\cap((a\cap c)\sqcup b)$.  But $a\cap b,a\cap c\leq a$ so $(a\cap b)\sqcup (a\cap c)=(a\cap c)\sqcup (a\cap b)$ by the first part of the proof.  The result now follows.
\end{proof}

\begin{lem} \label{critical}
Suppose $(S,\sqcup,\cap)$ is an ado-semilattice.  If $a,b\in S$ have an upper bound and $d\lesssim a, d\lesssim b$, then $d\lesssim a\cap b$.
\end{lem}
\begin{proof}
Assume $a,b\in S$ have an upper bound and $d\lesssim a, d\lesssim b$.  We prove the following facts in turn. 
\ben
\item $b\leq (b\cap (d\sqcup b))\sqcup a$;
\item $b\sqcup a= (b\cap (d\sqcup b))\sqcup a$;
\item $b\leq (a\cap b)\sqcup d\sqcup b$;
\item $d\lesssim a\cap b$.
\een
Since $d\lesssim a, d\lesssim b$, we obtain $a\sqcup d=a, b\sqcup d=b$.
For (1), we have that $b\sqcup a=a\sqcup b$ by (1) in Lemma \ref{goodies}, and so from (3) in Corollary \ref{altado},
\[b\leq a\sqcup b=(a\sqcup b)\cap(a\sqcup b)=(a\sqcup b)\cap(a\sqcup d\sqcup b)=a\sqcup (b\cap(d\sqcup b))
=(b\cap(d\sqcup b))\sqcup a,\] 
again using (1) in Lemma \ref{goodies} (since $a,b\cap(d\sqcup b)$ have an upper bound, namely any upper bound of $a,b$).
For (2), first note that
\begin{align*}
(b\sqcup a)\cap ((b\cap (d\sqcup b))\sqcup a)
&= (b\sqcup a)\cap ((b\cap (d\sqcup b))\sqcup a\sqcup b)\mbox{ (by (1))}\\
&= (b\sqcup a)\cap ((b\cap (d\sqcup b))\sqcup b\sqcup a)\\
&=(b\sqcup a)\cap (b\sqcup a)\\
&\mbox{(since $b\cap (d\sqcup b)\leq b$)}\\
&=b\sqcup a,
\end{align*}
so $b\sqcup a\leq ((b\cap (d\sqcup b))\sqcup a$.  Hence
\begin{align*}
(b\cap (d\sqcup b))\sqcup a)&=(b\sqcup a)\sqcup((b\cap (d\sqcup b))\sqcup a))\\
&=b\sqcup a\sqcup (b\cap(d\sqcup b))\\
&=b\sqcup a\sqcup b\sqcup (b\cap(d\sqcup b))\\
&=b\sqcup a\sqcup b\\
&\mbox{(since $b\cap(d\sqcup b)\lesssim b$ by (4) in Lemma \ref{useful})}\\
&=b\sqcup a.
\end{align*}
For (3),
\begin{align*}
b\cap ((a\cap b)\sqcup d\sqcup b)
&=b\cap((b\cap(d\cap b))\sqcup a)\\
&\mbox{(by (2) in Lemma \ref{goodies})}\\
&=b\cap(b\sqcup a)\mbox{ (by (2))}\\
&=b,
\end{align*}
so $b\leq (a\cap b)\sqcup d\sqcup b$.
Finally, for (4):
\begin{align*}
(a\cap b)\sqcup d\sqcup b&=b\sqcup(a\cap b)\sqcup d\sqcup b\mbox{ (by (3))}\\
&=b\sqcup d\sqcup b\\
&=b\sqcup d\\
&=b,
\end{align*}
so $b\geq (a\cap b)\sqcup d$.  Similarly, $a\geq (b\cap a)\sqcup d=(a\cap b)\sqcup d$.  So 
\[(a\cap b)\sqcup d\leq a\cap b\leq (a\cap b)\sqcup d,\]
so $a\cap b=(a\cap b)\sqcup d$, and so $d\lesssim a\cap b$.
\end{proof}

We may now infer two quasiequations for ado-semilattices.

\begin{lem}  \label{qe}
In an ado-semilattice, the following two quasiequations hold:
\ben
\item $d\lesssim a\cap b, d\lesssim b\cap c\Rightarrow d\lesssim a\cap c$;
\item $d\lesssim a, d\lesssim b, d\lesssim a\curlyvee b\Rightarrow d\lesssim a\cap b$.
\een
\end{lem}
\begin{proof}
Suppose $d\lesssim a\cap b, d\lesssim b\cap c$.  Then by Lemma \ref{critical}, we have that 
\[
d\lesssim (a\cap b)\cap (b\cap c)=a\cap b\cap c\leq a\cap c,
\] 
and so $d\lesssim a\cap c$.

Now suppose $d\lesssim a, d\lesssim b, d\lesssim a\curlyvee b$. But $a,a\curlyvee b\leq a\sqcup b$, so $d\lesssim a\cap (a\curlyvee b)$ by Lemma \ref{critical}.  Similarly, $d\lesssim b\cap (a\curlyvee b)$.  Using Lemma \ref{critical} again, we obtain $d\lesssim a\cap b\cap (a\curlyvee b)\leq a\cap b$, so $d\lesssim a\cap b$ by (4) in Lemma \ref{useful}. 
\end{proof}

\subsection{$\curlyvee$-algebras}

On any ado-semilattice $S$, we define the operation $\vee$ as follows: $a\vee b=(a\sqcup b)\cap (b\sqcup a)$ for all $a,b\in S$.  

\begin{pro}
In a functional ado-semilattice, $a\vee b=a\curlyvee b$.  \label{veevee}
\end{pro}
\begin{proof}
First, note that $a\curlyvee b\subseteq a\sqcup b,b\sqcup a$, so $a\curlyvee b\subseteq (a\sqcup b)\cap (b\sqcup a)$.  Conversely, suppose $(x,y)\in (a\sqcup b)\cap (b\sqcup a)$, so $(x,y)\in a\sqcup b$ and $(x,y)\in b\sqcup a$.  If $(x,y)\in a,b$, then $(x,y)\in a\cap b\subseteq a\curlyvee b$.  So assume $(x,y)\not\in a$.  Then because $(x,y)\in a\sqcup b$, $(x,y)\in b$ and $x\not\in \dom(a)$, so $(x,y)\in a\curlyvee b$.  Similarly if $(x,y)\not\in b$, $(x,y)\in a\curlyvee b$.  So  $(a\sqcup b)\cap (b\sqcup a)\subseteq a\curlyvee b$.
\end{proof}

So from now on, we use the notation $\curlyvee$ rather than $\vee$ for this derived operation on an ado-semilattice.  A number of laws involving $\curlyvee$ only may be derived.

\begin{pro}  \label{curlyveeprops}
Let $S$ be an ado-semilattice.  Then for all $a,b,c\in S$:
\ben
\item $a\sqcup b=a\curlyvee(a\curlyvee b)$;
\item $\curlyvee$ is commutative and idempotent;
\item $(a\curlyvee b)\sqcup (a\sqcup b)=(a\sqcup b)$;
\item $a\sqcup(b\curlyvee c)=(a\sqcup b)\curlyvee (a\sqcup c)$;
\item $d\lesssim a, d\lesssim b, d\lesssim c, d\lesssim a\curlyvee b$ and $d\lesssim b\curlyvee c$ imply that $d\lesssim a\curlyvee c$.
\een
\end{pro}
\begin{proof}
Now $a,a\curlyvee b\leq a\sqcup b$, so by (1) in Lemma \ref{goodies}, $a\sqcup(a\curlyvee b)=(a\curlyvee b)\sqcup a$, and so $a\curlyvee(a\curlyvee b)=a\sqcup(a\curlyvee b)=a\sqcup ((a\sqcup b)\cap(b\sqcup a))=(a\sqcup a\sqcup b)\cap (a\sqcup b\sqcup a)=(a\sqcup b)\cap (a\sqcup b)=a\sqcup b$ as required.  Clearly $\curlyvee$ is commutative and idempotent.
Further, 
\begin{align*}
(a\sqcup b)\curlyvee (a\sqcup c)
&=((a\sqcup b)\sqcup(a\sqcup c))\cap((a\sqcup c)\sqcup (a\sqcup b))\\
&=(a\sqcup (b\sqcup c))\cap(a\sqcup (c\sqcup b))\\
&=a\sqcup((b\sqcup c)\cap (c\sqcup b))\\
&=a\sqcup (b\curlyvee c),
\end{align*}
establishing the final equational law.

Finally, if $d\lesssim a,d\lesssim b, d\lesssim c, d\lesssim a\curlyvee b$ and $d\lesssim b\curlyvee c$, then using (2) in Lemma \ref{qe} twice, we obtain that $d\lesssim a\cap b$ and $d\lesssim b\cap c$, and so using (1) in Lemma \ref{qe}, $d\lesssim a\cap c$, and so $d\lesssim a\curlyvee c$.  So $(S,\curlyvee,\cap)$ is a $\curlyvee$-algebra.  
\end{proof}

The properties just listed turn out to be critical ones.

\begin{dfn}
Let $S$ be an algebra with binary operation $\curlyvee$ such that:
\ben
\item if we set $a\sqcup b=a\curlyvee (a\curlyvee b)$ for all $a,b\in S$, $(S,\sqcup)$ is a left regular band, and
\item $(S,\curlyvee)$ satisfies the properties listed in Proposition \ref{curlyveeprops}.
\een  
Then we say $S$ is a {\em $\curlyvee$-algebra}.
\end{dfn}

The final law for $\curlyvee$-algebras (in Proposition \ref{curlyveeprops}) can be expressed as an algebraic quasiequation, so the class of $\curlyvee$-algebras is a quasivariety. 

\begin{cor}  \label{faca}
Every ado-semilattice is a $\curlyvee$-algebra when we define $a\curlyvee b=(a\sqcup b)\cap (b\sqcup a)$.  Hence every functional algebra of signature $(\curlyvee)$ is a $\curlyvee$-algebra.
\end{cor}
\begin{proof}
If $S$ is an ado-semilattice and we define $\curlyvee$ as indicated, then we know that $a\curlyvee(a\curlyvee b)$ gives $a\sqcup b$ which is a left regular band operation.  The other $\curlyvee$-algebra properties follow from Proposition \ref{curlyveeprops}.

Since the $\curlyvee$-algebra laws are derivable from the ado-semilattice laws which are functionally sound, and for functional ado-semilattices, $a\curlyvee b=(a\sqcup b)\cap (b\sqcup a)$, it follows that every functional algebra of signature $(\curlyvee)$ is a $\curlyvee$-algebra.
\end{proof}  

The quasiequation in the definition of $\curlyvee$-algebras cannot be replaced by equations.

\begin{thm}  \label{proper}
The class of $\curlyvee$-algebras is a proper quasivariety. 
\end{thm}
\begin{proof}
Let $X=\{a,b\}$ and consider $S=\{0,1,i,1_a,1_b,p_a,1\}\subseteq \Par(X,X)$, where $0=\emptyset, 1=\{(a,a),(b,b)\}, i=\{(b,a)\}, 1_a=\{(a,a)\}, 1_b=\{(b,b)\}$ and $p_a=\{(a,a),(b,a)\}$.
Then $S$ is a $\curlyvee$-subalgebra of $\Par(X,X)$ in which the tables for $\curlyvee$ and the derived operation $\sqcup$ are easily seen to be as follows.

\[
\begin{array}{c|cccccc}
\curlyvee&0&i&1_b&1_a&p_a&1\\
\hline
0&0&i&1_b&1_a&p_a&1\\
i&i&i&0&p_a&1_a&1_a\\
1_b&1_b&0&1_b&1&1_a&1\\
1_a&1_a&p_a&1&1_a&p_a&1\\
p_a&p_a&p_a&1_a&p_a&p_a&1_a\\
1&1&1_a&1&1&1_a&1
\end{array}
\hspace{1.5cm}
\begin{array}{c|cccccc}
\sqcup&0&i&1_b&1_a&p_a&1\\
\hline
0&0&i&1_b&1_a&p_a&1\\
i&i&i&i&p_a&p_a&p_a\\
1_b&1_b&1_b&1_b&1&1&1\\
1_a&1_a&p_a&1&1_a&p_a&1\\
p_a&p_a&p_a&p_a&p_a&p_a&p_a\\
1&1&1&1&1&1&1
\end{array}
\]
\bigskip

It is easy to see that the equivalence relation $\theta$ in which the $\theta$-classes are $\{0\}$, $\{i\}$, $\{1_b\}$, and $\{1_a,p_a,1\}$ is a congruence on $S$.  Denoting the non-singleton $\theta$-class by $1$ and the others by their only elements, the tables for $S/\theta$ are evidently as follows.

\[
\begin{array}{c|cccc}
\curlyvee&0&i&1_b&1\\
\hline
0&0&i&1_b&1\\
i&i&i&0&1\\
1_b&1_b&0&1_b&1\\
1&1&1&1&1
\end{array}
\hspace{2cm}
\begin{array}{c|cccc}
\sqcup&0&i&1_b&1\\
\hline
0&0&i&1_b&1\\
i&i&i&i&1\\
1_b&1_b&1_b&1_b&1\\
1&1&1&1&1
\end{array}
\]
\bigskip

It is easy to see that $1_b\lesssim 1_b, 1_b\lesssim 1, 1_b\lesssim i, 1_b\lesssim 1_b\curlyvee 1, 1_b\lesssim 1\curlyvee i$, and yet $1_b\lesssim 1_b\curlyvee i$ fails since $(1_b\curlyvee i)\sqcup 1_b=1_b=1_b\neq 0=1_b\curlyvee i$.  

So the quasiequational law for $\curlyvee$-algebras does not hold and so the class of $\curlyvee$-algebras is not closed under taking homomorphic images.  
\end{proof}

We do not pretend that the axioms used to define $\curlyvee$-algebras are irredundant.  Indeed it can be shown using {\em Prover9} that associativity of $\sqcup$, idempotence of $\curlyvee$, and the third law in Proposition \ref{curlyveeprops} are indeed all redundant.  It can also be shown that the left regular band partial order may be expressed directly in terms of $\curlyvee$, via $a\leq b$ if and only if $a\curlyvee b=b$.    However, none of this affects either finite axiomatisabilty or the proper quasiequational nature of these axioms, so we do not give detailed proofs. 

Given the above comments, along with the obvious fact that the defining axioms for $\curlyvee$-algebras could easily be expressed entirely in terms of $\curlyvee$ (using the first law in Proposition \ref{curlyveeprops}), one might hold out hope of obtaining a general theory of ``generalised $\curlyvee$-algebras" that can be viewed as part of the general theory of partially ordered sets, as was done with o-semilattices in \cite{cirulis}.  There, it was shown that a poset is an o-semilattice in at most one way.  However, in Example \ref{difflrb}, a poset $(S,\leq)$ that is even a Boolean distributive lattice arises from two non-isomorphic left regular band structures on $S$, $(S,\vee)$ and $(S,\sqcup)$.  For each, one may wish to define $a\curlyvee_1 b=(a\vee b)\wedge (b\vee a)$ and $a\curlyvee_2 b=(a\sqcup b)\wedge (b\sqcup a)$, but then $e\curlyvee_1 f=t$ whereas $e\curlyvee_2 f=b$.  

In this last example, $(S,\curlyvee_1)$ is a $\curlyvee$-algebra, but $(S,\curlyvee_2)$ is not.  However, even if we consider only $\curlyvee$-algebras, a similar issue arises.  

\begin{eg} 
Two non-isomorphic $\curlyvee$-algebra structures on a set $S$ giving the same partial order on $S$.  \label{fixord}

{\em Let $S=\{a,b,c,d,0,f\}$.  We define two $\curlyvee$-algebra structures on $S$, having ``multiplication" tables for $\curlyvee$ and hence $\sqcup$ given below.  In the first (yielding $\curlyvee_1,\sqcup_1$), let $X=\{x,y\}$ and in $\Par(X,X)$, let $a=\{(x,x)\}$, $b=\{(x,y)\}$, $c=\{(x,x),(y,y)\}$, $d=\{(x,y),(y,y)\}$, $0=\emptyset$ (the empty function), and $f=\{(y,y)\}$.  It is straightforward to check that $S$ is then closed under $\curlyvee$ and yields the multiplication table for $\curlyvee_1$ given below, and similarly for $\sqcup_1$.  Likewise, one can easily check that the following definitions of the members of $S$ in $\Par(X,X)$ where $X=\{x,y,z\}$ yield the tables below for $\curlyvee_2$ and $\sqcup_2$: $a=\{(x,x),(y,y)\}$, $b=\{(y,z),(z,z)\}$, $c=\{(x,x),(y,y),(z,z)\}$, $d=\{(x,x),(y,z),(z,z)\}$, $0=\emptyset$, and $f=\{(x,x),(z,z)\}$.  

\[
\begin{array}{c|cccccc}
\curlyvee_1&a&b&c&d&0&f\\
\hline
a&a&0&c&f&a&c\\
b&0&b&f&d&b&d\\
c&c&f&c&f&c&c\\
d&f&d&f&d&d&d\\
0&a&b&c&d&0&f\\
f&c&d&c&d&f&f
\end{array}
\hspace{2cm}
\begin{array}{c|cccccc}
\sqcup_1&a&b&c&d&0&f\\
\hline
a&a&a&c&c&a&c\\
b&b&b&d&d&b&d\\
c&c&c&c&c&c&c\\
d&d&d&d&d&d&d\\
0&a&b&c&d&0&f\\
f&c&d&c&d&f&f
\end{array}
\]
and
\[
\begin{array}{c|cccccc}
\curlyvee_2&a&b&c&d&0&f\\
\hline
a&a&f&c&f&a&c\\
b&f&b&f&d&b&d\\
c&c&f&c&f&c&c\\
d&f&d&f&d&d&d\\
0&a&b&c&d&0&f\\
f&c&d&c&d&f&f
\end{array}
\hspace{2cm}
\begin{array}{c|cccccc}
\sqcup_2&a&b&c&d&0&f\\
\hline
a&a&c&c&c&a&c\\
b&d&b&d&d&b&d\\
c&c&c&c&c&c&c\\
d&d&d&d&d&d&d\\
0&a&b&c&d&0&f\\
f&c&d&c&d&f&f
\end{array}
\]
\bigskip

The only differences here are that $a\curlyvee_1 b=b\curlyvee_1 a=0$, and $a\sqcup_1 b =a$ and $b\sqcup_1 a=b$, whereas $a\curlyvee_2 b=b\curlyvee_2 a=f$, and $a\sqcup_2 b =c$ and $b\sqcup_2 a=d$.   In both cases, we see that $a\leq c, b\leq d, 0\leq s$ for all $s\in S$, and $f\leq c,d$, and it follows that in both, $(S,\leq)$ is a meet-semilattice with the following table:
\[
\begin{array}{c|cccccc}
\wedge&a&b&c&d&0&f\\
\hline
a&a&0&a&0&0&0\\
b&0&b&0&b&0&0\\
c&a&0&c&f&0&f\\
d&0&b&f&d&0&f\\
0&0&0&0&0&0&0\\
f&0&0&f&f&0&f
\end{array}
\]
For the first representation of $S$, $\wedge$ as above is intersection, as is easily seen, so $(S,\curlyvee_1,\wedge)$ is an ado-semilattice and indeed the unique ado-semilattice defined on the semilattice $(S,\wedge)$.  Hence $(S,\curlyvee_2,\wedge)$ cannot be an ado-semilattice, and indeed one sees that $a\wedge d=0$ does not correctly calculate intersection of $a,d$ in the second representation.
}
\end{eg}

So one cannot view the study of $\curlyvee$-algebras as simply part of the study of partially ordered sets, the way one can so view the study of lattices, Heyting algebras, o-semilattices and so on, since the operation $\curlyvee$ can admit no partial order-theoretic description.

\section{The completeness results}  \label{sec:complete}

Our main task in the remainder is to prove the converse of the second statement in Corollary \ref{faca}.  From this, the functional completeness of the ado-semilattice axioms will also follow.

\subsection{Completeness for $\curlyvee$-algebras}

In \cite{cirulis}, the notion of a flat o-semilattice was critical, and the same applies here.  It is easy to see that a flat poset can be turned into a $\curlyvee$-algebra.

\begin{dfn}
A {\em flat $\curlyvee$-algebra} $S$ is one in which there is $0\in S$ such that $a\curlyvee b=0$ if $a,b\in S$ are unequal and $a,b\neq 0$, with $a\curlyvee a=a$ and $a\curlyvee 0=0\curlyvee a=a$ for all $a\in S$.
\end{dfn}

Given a flat poset, it can easily be turned into a flat $\curlyvee$-algebra by setting $a\curlyvee b=0$ if $a,b\neq 0$, with $a\curlyvee 0=a$ for all $a$, and every flat $\curlyvee$-algebra arises in this way.

The following is the relevant variant of Lemma 5.1 in \cite{cirulis}.

\begin{pro}  \label{flatsimp}
Every flat $\curlyvee$-algebra is simple, hence subdirectly irreducible.
\end{pro}
\begin{proof}
Suppose $S$ is a flat $\curlyvee$-algebra, with $\psi:S\rightarrow T$ a non-injective homomorphism into another $\curlyvee$-algebra $T$.
So there are unequal $a,b\in S$ for which $a\psi=b\psi$, and we can assume without loss of generality that $a\neq 0$.  If $b=0$ then $a\curlyvee b=a$ so $a\psi=0\psi$, whereas if $b\neq 0$, then $a\psi=a\psi\curlyvee a\psi=a\psi\curlyvee b\psi=(a\curlyvee b)\psi=0\psi$.  So in all cases, $a\psi=0\psi$.  Then for any non-zero $c\in S$, $c\psi=(c\curlyvee 0)\psi=c\psi\curlyvee 0\psi=c\psi\curlyvee a\psi=(c\curlyvee a)\psi=0\psi$.  Hence $\psi$ is a constant.  This shows $S$ is simple.
\end{proof}

\begin{pro}  \label{isfunc}
Every flat $\curlyvee$-algebra is functional, and a direct product of functional algebras of signature $(\curlyvee)$ is functional.
\end{pro}
\begin{proof}
Let $S$ be a flat $\curlyvee$-algebra with bottom element $0$, and let $X=\{0\}$ and $Y=S$.  For each $a\in S$, define $\psi_a\in \Par(X,Y)$ by setting $\psi_0=\emptyset$ and $\psi_a=\{(0,a)\}$.  It is straightforward to check that the map $a\mapsto \psi_a$ is a $\curlyvee$-embedding.  

Consider a family $S_i,i\in {\mc I}$ of functional $\curlyvee$-algebras $S_i$.  View $S_i$ as embedded in $\Par(X_i,Y_i)$ where it is assumed that the $X_i$ are all disjoint from one-another, and similarly for the $Y_i$, let $X=\bigcup_i X_i$ and $Y=\bigcup_i Y_i$, and then note that $\curlyvee$ on the direct product corresponds to $\curlyvee$ on the corresponding elements of $\Par(X,Y)$.
\end{proof}

Hence any subalgebra of a direct product is also functional, and in particular, any subdirect product of flat algebras is functional.  We shall show that the converse is true: every functional $\curlyvee$-algebra is a subdirect product of flat algebras.  This will follow immediately from our proof that every $\curlyvee$-algebra is a subdirect product of flat algebras.  

First, we note that the quasiequation in the definition of $\curlyvee$-algebras may be usefully strengthened.

\begin{lem}  \label{bigqi}
If $S$ is a $\curlyvee$-algebra, then the following law holds: $d\lesssim a\sqcup i,d\lesssim b\sqcup i,d\lesssim c\sqcup i,d\lesssim (a\curlyvee b)\sqcup i, d\lesssim (b\curlyvee c)\sqcup i$ imply that $d\lesssim (a\curlyvee c)\sqcup i$.
\end{lem}
\begin{proof}
Suppose $d\lesssim a\sqcup i,b\sqcup i,c\sqcup i,(a\curlyvee b)\sqcup i,(b\curlyvee c)\sqcup i$.  Then $d\lesssim i\sqcup a,i\sqcup b,i\sqcup c,i\sqcup (a\curlyvee b)=(i\sqcup a)\curlyvee(i\sqcup b),i\sqcup (b\curlyvee c)=(i\sqcup b)\curlyvee(i\sqcup c)$, and we may infer from the simplified quasiequation that $d\lesssim (i\sqcup a)\curlyvee(i\sqcup c)=i\sqcup(a\curlyvee c)\lesssim (a\curlyvee c)\sqcup i$.  So the general quasiequation holds.  
\end{proof}

\begin{pro}  \label{curlid} 
A subset $I$ of a $\curlyvee$-algebra is a $\lesssim$-ideal if and only if it is a down-set under $\lesssim$ such that $i\curlyvee j\in I$ whenever $i,j\in I$.
\end{pro}
\begin{proof} This follows because $i\curlyvee j\leq i\sqcup j$ and so $i\curlyvee j\lesssim i\sqcup j$, and also $i\sqcup j=i\curlyvee(i\curlyvee j)$.
\end{proof}
Now let $I$ be a relatively maximal $\lesssim$-ideal of the $\curlyvee$-semigroup $S$.  Let
\[
E_I=\{(a,b)\in S\times S\mid a\not\in I, b\not\in I, a\curlyvee b\not\in I\},\mbox{ and }
\epsilon_I=E_I\cup (I\times I).
\]

\begin{lem}  \label{biggie}
Let $S$ be a $\curlyvee$-algebra, with $I$ a relatively maximal $\lesssim$-ideal of $S$.  Then $\epsilon_I$ is a congruence on $S$ and $S_I=S/\epsilon_I$ is a flat $\curlyvee$-algebra.  Moreover $S$ is a subdirect product of the $S_I$ as $I$ ranges over all relatively maximal $\lesssim$-ideals of $S$.
\end{lem} 
\begin{proof} 
Suppose $I$ is relatively maximal with respect to not containing $d\in S$.  Clearly, $\epsilon_I$ is reflexive and symmetric.  For transitivity, suppose $(a,b)\in \epsilon_I$ and $(b,c)\in \epsilon_I$.   If $a,b\in I$ and $b,c\in I$ then of course $a,c\in I$.  The only other case to check is whether none of $a,b,c,a\curlyvee b,b\curlyvee c$ being in $I$ implies also $a\curlyvee c\not\in I$.   

Suppose for a contradiction that $a\not\in I, b\not\in I, c\not\in I, a\curlyvee b\not\in I, b\curlyvee c\not\in I$, yet $a\curlyvee c\in I$.  Then by Lemma \ref{relmax}, there are $i_1,i_2,\ldots,i_5$ such that $d\lesssim a\sqcup i_1$, $d\lesssim b\sqcup i_2$, $d\lesssim c\sqcup i_3$, $d\lesssim (a\curlyvee b)\sqcup i_4$ and $d\lesssim (b\curlyvee c)\sqcup i_5$.  So by (5) in Lemma \ref{useful}, $d\lesssim a\sqcup i$, $d\lesssim b\sqcup i$, $d\lesssim c\sqcup i$, $d\lesssim (a\curlyvee b)\sqcup i$ and $d\lesssim (b\curlyvee c)\sqcup i$, where $i=i_1\sqcup\ldots\sqcup i_5\in I$.  So by Lemma \ref{bigqi}, we have that $d\lesssim (a\curlyvee c)\sqcup i\in I$, so $d\in I$, a contradiction.  So $a\curlyvee c\not\in I$, as required.  Hence $\epsilon_I$ is an equivalence relation.

Let $\overline{a}$ denote the $\epsilon_I$-class of $S_I$ containing $a\in S$.  On $S_I$, define the flat $\curlyvee$-algebra structure obtained by making $I$ the bottom element with all other $\epsilon_I$-classes maximal.  Define $\theta_I:S\rightarrow S_I$ by setting 
\[
a\theta_I=\begin{cases} \overline{a}&\mbox{ if }a\not\in I\\ I&\mbox{ if }a\in I.
\end{cases}
\]
We shall show that $\theta_I$ is a (clearly surjective) homomorphism; since its kernel is $\epsilon_I$, it will follow that $S_I\cong S/\epsilon_I$.  We perform a case analysis.  Now by Proposition \ref{curlid}, for $a,b\in S$, $a,b\in I$ imply $a\curlyvee b\in I$, and if $a,a\curlyvee b\in I$ then $b\in I$ (since $b\leq b\sqcup a\lesssim a\sqcup b=a\curlyvee (a\curlyvee b)\in I$). 

Suppose $a\curlyvee b\in I$ but $a\not\in I,b\not\in I$.  Then 
\begin{align*}
a\theta_I\curlyvee b\theta_I&=\ol{a}\curlyvee\ol{b}\\
&=I\mbox{ since $a\curlyvee b\in I$ so $\ol{a}\neq \ol{b}$}\\
&=(a\curlyvee b)\theta_I\mbox{ since $a\curlyvee b\in I$.}
\end{align*}

Suppose $a\curlyvee b\in I$ and $a\in I,b\in I$.  Then $a\theta_I\curlyvee b\theta_I=I\curlyvee I=I=(a\curlyvee b)\theta_I$.

Suppose $a\curlyvee b\not\in I$, $a\not\in I,b\not\in I$.  Then $\ol{a}=\ol{b}$, and $a\curlyvee(a\curlyvee b)=a\sqcup b\geq a\curlyvee b\not\in I$, so $a\curlyvee(a\curlyvee b)\not\in I$, and so $\ol{a}=\ol{a\curlyvee b}$.  Hence $a\theta_I\curlyvee b\theta_I=\ol{a}\curlyvee \ol{b}=\ol{a}=\ol{a\curlyvee b}=(a\curlyvee b)\theta_I$.

Suppose $a\curlyvee b\not\in I, a\not\in I, b\in I$.  Then as in the previous case, $\ol{a}=\ol{a\curlyvee b}$, and so $a\theta_I\curlyvee b\theta_I=\ol{a}\curlyvee I=\ol{a}=\ol{a\curlyvee b}=(a\curlyvee b)\theta_I$.  The case $a\curlyvee b\not\in I, a\in I, b\not\in I$ is very similar by symmetry.

This covers all cases, so $\theta_I$ is a surjective homomorphism, and because $ker(\theta_I)=\epsilon_I$, the latter must be a congruence on $S$.
It now suffices to show that $S$ is a subdirect product of the $S_I$, for which it suffices to show that $\epsilon=\bigcap_I \epsilon_I$ is the diagonal relation.  

Let $a,b\in S$ be distinct, and without loss of generality assume $a\leq b$, so $a\lesssim b$.  We shall find a relatively maximal $\lesssim$-ideal $I$ such that $a\theta_I\neq b\theta_I$, so that $(a,b)\not\in \epsilon_I$ and so $(a,b)\not\in \epsilon$.  

We consider two cases.  First suppose that $b\not\lesssim a$.  Then let $I_0=\{s\in S\mid s\lesssim a\}$, a $\lesssim$-ideal since it is obviously down-closed under $\lesssim$, and if $i,j\in I_0$ then $a\sqcup i=a=a\sqcup j$ so $a\sqcup (i\sqcup j)=(a\sqcup i)\sqcup j=a\sqcup j=a$, and so $i\sqcup j\lesssim a$ also.  Moreover $a\in I_0$ but $b\not\in I_0$.  Extend $I_0$ to a $\lesssim$-ideal $I$ of $S$ that is maximal with respect to not containing $b$ using Zorn's Lemma; then of course $a\in I$ still.  Then $a\theta_I=I$ whereas $b\theta_I=\ol{b}\neq I$, so $(a,b)\not\in \epsilon_I$.

Secondly, suppose $b\lesssim a$, and recall that $a\lesssim b$.  Suppose for a contradiction that $b\lesssim a\curlyvee b$ also.  So $b\sqcup a=b$, $a\sqcup b=a$ and $(a\curlyvee b)\sqcup b=a\curlyvee b$.  Hence
\ben
\item $b\curlyvee(b\curlyvee a)=b$,
\item $a\curlyvee(a\curlyvee b)=a$,
\item $(a\curlyvee b)\curlyvee ((a\curlyvee b)\curlyvee b)=a\curlyvee b$.
\een
Recalling that $s\curlyvee t=t\curlyvee s$ for all $s,t$, using (1) in (3) gives $(a\curlyvee b)\curlyvee b=a\curlyvee b$, and (1) then gives $b=a\curlyvee b$, and then (2) gives $a\curlyvee b=a$, and so $a=b$, a contradiction.  So $b\not\lesssim a\curlyvee b$.
Now let $I_0=\{s\in S\mid s\lesssim a\curlyvee b\}$, again a $\lesssim$-ideal of $S$, and $b\not\in I_0$ as just shown.  Again extend to a $\lesssim$-ideal of $S$ that is maximal with respect to not containing $b$ using Zorn's Lemma, noting that $a\curlyvee b\in I$.  If $a\in I$ then $a\theta_I=I$, whereas $b\theta_I=\ol{b}\neq I$ (as before).  If $a\not\in I$, so that $b\not\in I$ and so $a\theta_I=\ol{a}$ and $b\theta_I=\overline{b}$ are both defined, because $a\curlyvee b\in I$, we have $a\theta_I\neq b\theta_I$.  So $(a,b)\not\in \epsilon_I$ in either case.
\end{proof}

From the above lemma, as well as Proposition \ref{isfunc} and the discussion after it, we obtain the following.

\begin{thm} \label{t2}
The class of functional algebras of signature $(\curlyvee)$ equals the class of $\curlyvee$-algebras, which consists of subdirect products of flat $\curlyvee$-algebras.
\end{thm}

This shows that the class of functional algebras of signature $(\curlyvee)$ is a finitely based quasivariety (and not a variety by Theorem \ref{proper}).  It also follows that knowledge of $\sqcup$ is enough to determine $\curlyvee$ in a $\curlyvee$-algebra, since from it one can recover $\leq$, hence meets where they exist, and then $a\curlyvee b=(a\sqcup b)\wedge(b\sqcup a)$ (a meet which will always exist in a $\curlyvee$-algebra since it is functional).

\subsection{Completeness for ado-semilattices}

The notion of flatness for ado-semilattices is precisely that for o-semilattices as discussed in \cite{cirulis}; from the poset perspective they are just flat posets.  As noted in \cite{cirulis}, all flat o-semilattices are functional, and hence are ado-semilattices, and indeed they are exactly o-semilattices that satisfy $0\cap a=a$ and $a\cap b=0$ for $a\neq b$; it follows that $a\sqcup b=a$ if $a\neq 0$, with $0\sqcup a=a$ for all $a$.

Since every ado-semilattice is a $\curlyvee$-algebra by Corollary \ref{faca}, we obtain the following consequence of Proposition \ref{flatsimp}.

\begin{cor}
Every flat ado-semilattice is simple, hence subdirectly irreducible.
\end{cor}

\begin{lem}  \label{biggie2}
Let $S$ be an ado-semilattice, with $I$ a relatively maximal $\lesssim$-ideal of $S$.  Then the $\curlyvee$-algebra congruence $\epsilon_I$ is an ado-semilattice congruence on $S$ and $S/\epsilon_I$ is a flat ado-semilattice.  Moreover $S$ is a subdirect product of the $S/\epsilon_I$ as $I$ ranges over all relatively maximal $\lesssim$-ideals of $S$.
\end{lem} 
\begin{proof}
Since each $S_I$ as in the proof of Lemma \ref{biggie} is made into a flat poset with $I$ as the bottom element, it has flat ado-semilattice structure as well.  To show what is claimed, it suffices to show that each $\curlyvee$-algebra homomorphism $\theta_I$ respects $\cap$ as well.

Let $a,b\in S$.  First, suppose $a\cap b\not\in I$.  Then $a\curlyvee b\not\in I$ since $a\cap b\leq a\curlyvee b$, and also $a,b\not\in I$ since $a\cap b\leq a,b$, so $\overline{a}=\ol{b}$.  Moreover $(a\cap b)\curlyvee a=((a\cap b)\sqcup a)\cap (a\sqcup (a\cap b))=a\cap a=a\not\in I$ since $a\cap b\leq a$.   Alternatively, suppose $a\cap b\in I$; then $(a\cap b)\theta_I=I$.  Now if $a\in I$ or $b\in I$, $a\theta_I\cap b\theta_I=I$.  If $a\not\in I$ and $b\not\in I$, then $a\neq b$ as otherwise $a\cap b\not\in I$, a contradiction, so $a\theta_I\cap b\theta_I=\ol{a}\cap \ol{b}=I$.  So in either case, we have $(a\cap b)\theta_I=a\theta_I\cap b\theta_I$, as required.  
\end{proof}

Again, we immediately obtain the following.

\begin{thm} \label{t3}
The class of functional algebras of signature $(\sqcup,\cap)$ is finitely axiomatised as the variety of ado-semilattices.
\end{thm}

\section{Difference}   \label{sec:diff}

Recall that the operation of difference can be used to express intersection.  The functional signature $(\sqcup,\backslash)$ is axiomatised in \cite{cirulis} as the variety of associative subtraction o-semilattices: these are associative o-semilattices which satisfy several further laws involving difference only which ensure they are Boolean nearlattices (see Proposition 6.1 there).  An alternative axiomatisation building on ado-semilattices follows from the next definition and result.

\begin{dfn} \label{odef}
An algebra $(S,\sqcup,\backslash)$ such that $(S,\sqcup,\cap)$ is an ado-semilattice and the following three laws are satisfied is said to be an {\em od-algebra} (override-difference algebra): 
\ben
\item $(a\backslash b)\cap b=0$ (meaning that $(a\backslash b)\cap b$ is constant);
\item $(a\backslash b)\sqcup (a\cap b)=a$.
\een
(Here we define $a\cap b=a\backslash (a\backslash b)$.)
\end{dfn}

It is fairly clear that the od-algebra laws are all functionally sound and that we may interpret $0$ as the empty function.  

\begin{pro}  \label{adominus}
The functional algebras of signature $(S,\sqcup,\backslash)$ are axiomatised by the od-algebra laws.
\end{pro}
\begin{proof}
Soundness was noted earlier.
Conversely, suppose $(S,\sqcup,\backslash)$ satisfies the above conditions.  Then we know that its $(S,\sqcup,\cap)$-reduct is functional from Theorem \ref{t3}.  Moreover, using the three laws involving difference, we have that for $a\in S$, $a\backslash a\leq a$, so $a\backslash a=(a\backslash a)\cap a=0$, and so $0\sqcup a=(a\backslash a)\sqcup a=a$, so $0\leq a$ for all $a\in S$, and so $0\in I$ for every $\lesssim$-ideal $I$ of $S$.  So each $\theta_I$ maps $0$ to $I$ in $S_I$.  View $S_I$ as partial functions; then the bottom element $I$ can be taken to be the empty function.  Hence in the resulting functional representation of $S$ determined by $\theta$ (or indeed any in which all other operations and $0$ are correctly represented), the three laws for $\backslash$ above force $a\backslash b$ to be (1) disjoint with $b$, (2) contained in $a$, and (3) such that when its (necessarily disjoint) union with $a\cap b$ is taken, the result is all of $a$.  This forces $a\backslash b$ to be the set-theoretic difference of $a,b$.  (In the language of Section 3.1 of \cite{overup}, $\backslash$ is {\em abstractly definable} in terms of $\sqcup,\cap,0$.) 
\end{proof}

So the variety of od-algebras axiomatises the functional algebras of signature $(\sqcup,\backslash)$, and is therefore equal to the variety of associative subtractive o-semilattices as in \cite{cirulis}.

\section*{Acknowledgements}

I would like to acknowledge the considerable time-saving assistance of the {\em Prover9/Mace4} software \cite{Pr9}, which helped with finding both proofs and examples.

\noindent Tim Stokes\\
Department of Mathematics\\
University of Waikato\\
Hamilton 3216\\
New Zealand.\\
email: tim.stokes@waikato.ac.nz

\end{document}